\documentclass[a4paper,11pt]{article}
\usepackage[utf8]{inputenc}
\usepackage{amsmath, amsthm, amssymb}
\usepackage{natbib}
\usepackage{geometry}
\usepackage{color}
\usepackage{graphicx}

\newcommand{\NN}{\mathbb{N}}

\newcommand{\RR}{\mathbb{R}}
\newcommand{\ZZ}{\mathbb{Z}}

\DeclareMathOperator{\EE}{\mathbb{E}}
\DeclareMathOperator{\PP}{\mathbb{P}}
\DeclareMathOperator{\Var}{Var}
\DeclareMathOperator{\Cov}{Cov}

\newcommand{\bx}{\mathbf{x}}
\newcommand{\bX}{\mathbf{X}}
\newcommand{\by}{\mathbf{y}}

\newcommand{\calB}{\mathcal{B}}

\newtheorem{theorem}{Theorem}[section]
\newtheorem{lemma}[theorem]{Lemma}
\newtheorem{proposition}[theorem]{Proposition}
\newtheorem{corollary}[theorem]{Corollary}

\theoremstyle{definition}
\newtheorem{definition}[theorem]{Definition}
\newtheorem{remark}[theorem]{Remark}
\newtheorem{example}[theorem]{Example}

\newtheorem*{condition}{Condition}

\title{Ordinal Patterns in Clusters of Subsequent Extremes of Regularly Varying Time Series}
\author{Marco Oesting\thanks{University of Siegen, Department Mathematik, Walter-Flex-Str.~3, 57068 Siegen, Germany, Email: oesting@mathematik.uni-siegen.de} \; 
	and Alexander Schnurr\thanks{University of Siegen, Department Mathematik, Walter-Flex-Str.~3, 57068 Siegen, Germany, Email:  schnurr@mathematik.uni-siegen.de}}

\begin{document}
	
\maketitle

\begin{abstract}
In this paper, we investigate temporal clusters of extremes defined as 
subsequent exceedances of high thresholds in a stationary time series. Two 
meaningful features of these clusters are the probability distribution of the 
cluster size and the ordinal patterns giving the relative positions of the data points within a cluster. Since these patterns take only the ordinal structure of consecutive data points into account, the method is robust under monotone transformations and measurement errors. We verify the existence of the corresponding limit distributions in the framework of regularly varying time series, develop non-parametric estimators and show their asymptotic normality under appropriate mixing conditions. The performance of the estimators is demonstrated in a simulated example and a real data application to discharge data of the river Rhine.
\end{abstract}

\section{Introduction}

In time series data sets, extremes often do not occur at scattered instants of time, but tend to form clusters. Assigning a cluster of extremes to a single extreme event, such as a flood in the context of a hydrological time series or a stock market crash in the context of a financial data, the distribution of these clusters is crucial for risk assessment.

In order to analyze the occurrence times of extremes defined as exceedances over some high threshold $u$, some profound theory has been built up since the 1970s. Within this framework, data $X_1,\ldots,X_n$ from a stationary time series $(X_t)_{t \in \ZZ}$ are typically divided into different blocks. Then, repeated extremes are said to form a cluster if they occur within the same temporal block. Due to the convergence of the process of exceedances to a Poisson point process under appropriate conditions as $u \to \infty$, the distribution of these clusters converges weakly provided that the block size increases at the right speed. The limit distribution is nicely linked to the well-known concept of the extremal index of the time series which can be interpreted as the reciprocal of the mean limiting cluster size \citep[cf.][for an overview]{llr-83,ekm-97,chavez-davison-12}. 
Besides the extremal index, several other cluster characteristics are of interest and can be estimated, such as the distribution of the cluster size \citep{robert-09} or more general cluster functionals \citep{drees-rootzen-10}. Convergence of clusters in appropriate sequence 
spaces preserving the order of observations can be shown within the framework of regular variation \citep{BPS18}.

Even though positive theoretical results exist, estimation of characteristics of clusters as defined above is difficult for finite samples. Here, besides the threshold $u$, also the block size or, equivalently, some cluster identification parameter giving the minimum distance between two separate clusters, needs to be chosen. Instead, in
this paper, we will use a different definition of a cluster of extremes by restricting our attention to subsequent threshold exceedances, i.e.\ a realization of the $l$-dimensional vector $(X_i)_{i=t}^{t+l-1}$ will be called a $u$-exceedance cluster of size $l \in \NN$ if and only if 
\begin{equation} \label{eq:def-cluster}
X_{t-1} \leq u, X_t > u, \ldots,
X_{t+l-1}>u \text{ and } X_{t+l} \leq u.
\end{equation}
As any non-exceedance will separate two clusters, this definition is much stricter than the classical definition described above. An advantage of the definition of $u$-exceedance clusters is that it depends on one parameter, namely the threshold $u$, only. Such a cluster definition has 
already been employed in a series of papers by \citet{markovich-14, markovich-16, markovich-17} who analyzes the limit distribution of two cluster characteristics.
First, she considers the number of inter-cluster times $T_1(u)$, i.e.\ the number of observations between two consecutive clusters, which is a random variable with the same distribution as 
$$ \min\{j \geq 1: \, X_{j+1} > u\} \quad \text{conditional on } X_1 > u.$$ 
Note that this number of inter-cluster time also plays an important role in the estimation of the extremal index \citep{ferro-segers-03}.
Secondly, she studies the random variable $T_2(u)$ with the same distribution as
$$ \min\{j \geq 1: \, X_{j+1} \leq u\} \quad \text{conditional on } X_1 \leq u,$$  
i.e.\ $T_2(u)-1$ is the length of a $u$-exceedance cluster starting at some fixed
time. Since we have $\lim_{u \to \infty} \PP(X_2 \leq u) = 1$, the distribution of $T_2(u)$
is typically expected to converge weakly to a degenerate distribution, i.e.
$\lim_{u \to \infty} \PP(T_2(u) = 1) = 1$. In \citet{markovich-14, markovich-16},
under appropriate mixing conditions, the rate of convergence is determined as a function of the extremal index. More precisely, for all $\varepsilon > 0$,
there exist a threshold $u_0 = u_0(\varepsilon)$ a number $j_0 = j_0(\varepsilon)$ such that, for all $u > u_0$ and $j > j_0$,
$$ \left| \frac{1 - \PP(X_0>u)^\theta}{\theta^2 (1-\PP(X_0>u))} \cdot \frac{\PP(T_2(u) =j)}{[1-\PP(X_0>u)^\theta] \{\PP(X_0>u)^\theta\}^{j-1}} 
- 1 \right| < \varepsilon, $$
i.e.\ for a sufficiently large threshold $u$, the tail of the distribution of $T_2(u)$ becomes proportional to a geometric distribution with parameter
$1 - \PP(X_0 > u)^\theta$. Furthermore, \citet{markovich-14, markovich-16}
provides results for the duration of clusters if the time between subsequent observations is random.

In our paper, we will focus on the case of an equally spaced time series,
i.e.\ the case when the terms of the duration and the size of cluster coincide.
Here, instead of considering the probability that there is a cluster of a specific size at a certain time, we analyze the size of a randomly chosen $u$-exceedance cluster or, equivalently, we examine the size of a cluster conditional on being a cluster of positive length. Thus, we first address the question:
\begin{center}
	How long does an extreme event in a time series last provided that it occurs?
\end{center}

Secondly, we analyze so-called ordinal patterns which we find in the above 
mentioned clusters of extremes. Ordinal patterns keep the ordinal information of the data only and, thus, describe their `up-and-down behaviour'. Here, the relative position of the data points $x_0,\ldots,x_{l-1}$ is encoded by a permutation $\pi$ on $\{0,\ldots,l-1\}$ such that 
$$ x_{\pi(0)} \geq x_{\pi(1)} \geq \ldots \geq x_{\pi(l-1)}. $$
Note that this permutation is unique if the data points $x_0,\ldots,x_{l-1}$
are pairwise distinct. The following precise definition also accounts for the  ties by keeping the order of the indices in this case.

\begin{definition} \label{def:ord-pattern}
	For $l \in \NN$, let $S_{l-1}$ be the set of permutations of
	$\{0,\ldots,l-1\}$. The \emph{$l$-ordinal pattern} is defined as the mapping
	$\Pi: \RR^l \to S_{l-1}$ that maps a vector $(x_i)_{i=0}^{l-1}$ to 
	the unique permutation $\pi$ satisfying
    $$ x_{\pi(0)} \geq x_{\pi(1)} \geq \ldots \geq x_{\pi(l-1)} $$
    and $\pi(\min\{i,j\}) < \pi(\max\{i,j\})$ if $x_i = x_j$ for $i \neq j$.
\end{definition}

Ordinal patterns have been introduced in order
to analyze noisy data sets which appear in medicine, neuroscience and finance \citep[cf.][]{bandt-pompe-02, keller-etal-07, sinn-etal-13}. They have already been used successfully in the estimation of the Hurst parameter \citep{sinn-keller-11}. Further applications include tests for structural breaks \citep{sinn-etal-12} and the analysis of the Kolmogorov-Sinai entropy of dynamical systems \citep{keller-etal-15}.

Ordinal patterns can be used nicely to capture stylized facts as trends or inversions of the direction which might be used to characterize and classify certain events. To our knowledge the present paper is the first approach to analyze the ordinal behavior which can be observed in clusters of extremes of time series. The advantages of the proposed method include that the whole analysis is stable under monotone transformations of the state space. This will be useful in our analysis. Furthermore, the ordinal structure is not destroyed by small perturbations of the data or by measurement errors. There are fast algorithms to analyze the relative frequencies of ordinal patterns in given data sets \citep[cf.][Section 1.4]{keller-etal-07}. 

In the future, ordinal patterns in clusters of extremes might be used to detect structural breaks in the extremes of the given time series \citep[cf.][]{unakafov-keller-18}. Dealing with `correlated' time
series, one could analyze the dependence between extreme events in a non-linear fashion as it has been developed in \citet{schnurr-14} and 
\citet{schnurr-dehling-17}. This might be advantageous in particular if the time series are on totally different scales. Finally, as we will point out in Section \ref{sec:app}, ordinal patterns at the beginning of a cluster of extremes might be used in order to forecast the length
of this cluster in an on-line analysis of data.

Our analysis is embedded in a different theoretical framework than the works of
\citet{markovich-14, markovich-16, markovich-17}, namely, we will assume 
that the stationary time series of interest, $(X_t)_{t \in \ZZ}$, is regularly
varying. Note that this is a common assumption in extreme value theory allowing for convenient extrapolation to the tails of the distribution. More background on the
theory of regularly varying time series will be provided in Section \ref{sec:reg-var}. In Section \ref{sec:clust-dist}, we show that both the distribution of the size of $u$-exceedance clusters, as defined in \eqref{eq:def-cluster}, and
the distribution of the ordinal pattern within a cluster converge to 
(typically non-degenerate) limit distributions in case of a regularly
varying time series. Based on a sliding window approach, non-parametric 
empirical estimators for the limit distributions are introduced in Section 
\ref{sec:asymptotics}. Under conditions, similar to those considered in \citet{davis-mikosch-09} for the estimation of the extremogram, consistency
(Proposition \ref{prop:consistency}) and asymptotic normality (Corollary \ref{coro:asymptotic-normality}) of the estimators are established. In Section \ref{sec:simu-study}, we consider the example of max-stable time series and provide sufficient conditions in terms of extremal coefficients for Corollary \ref{coro:asymptotic-normality} to hold. The conditions are
verified for a Brown--Resnick time series which is then simulated to
demonstrate the finite-sample behaviour of the estimators. Finally,
we apply the estimator to daily discharge data of the river Rhine at Cologne in Section \ref{sec:app}. The proofs of our results can be found in the appendix.

\section{Background: Regular Varying Time Series} \label{sec:reg-var}

Throughout this paper, we will assume that $X = (X_t)_{t \in \ZZ}$ is a 
stationary time series whose marginal distribution $F_0$, defined by 
$F_0(x) = \PP(X_0 \leq x)$, is in the max-domain of attraction of an extreme
value distribution, i.e., there exist constants $a_n > 0$, $b_n \in \RR$, such
that
\begin{align*} %\label{eq:mda-uni}
F_0^n (a_n x + b_n) \stackrel{n \to \infty}{\longrightarrow} G_0(x), \quad x \in \RR,
\end{align*} 
for some non-degenerate distribution $G_0$. Without loss of generality, we may assume
that $G_0$ is an $\alpha$-Fr\'echet distribution for some $\alpha>0$,
$$ G_0(x) = \Phi_\alpha(x) = \exp(- x^{-\alpha}), \qquad x >0, $$
and that $F_0$ has a finite lower endpoint, $\inf\{x \in \RR: \, F_0(x)>0\} > - \infty$. 
Both properties can be achieved by applying strictly monotone marginal 
transformations to $(X_t)_{t \in \ZZ}$ provided that $F_0$ is continuous. As 
these transformations are the same for each $t \in \ZZ$ -- remind that $X$ is
stationary -- they do not have any effect on ordinal structure of the data.
In particular, ordinal patterns in extremes are invariant under these transformations.

A convenient framework for our further analysis will be provided by regular variation. Among several equivalent definitions of multivariate regular variation \citep[cf.][for instance]{resnick-07, resnick-08}, we will make use of the following convenient one \citep[cf.][for instance]{basrak-davis-mikosch-02}: We say that the $d$-variate random vector $\bX = (X_{t_1},\ldots,X_{t_d})$, $t_1,\ldots, t_d \in \ZZ$, 
is \emph{multivariate regularly varying} with index $\alpha>0$ if, for some 
norm $\|\cdot\|$ on $\RR^d$, there exists a probability measure $\sigma$ on the
sphere $\mathbb{S}^{d-1} = \{ \bx \in [0,\infty)^d: \, \|\bx\| = 1\}$ such that 
$$ \frac{ \PP\left(\|\bX\| > rx, \, \bX / \|\bX\| \in \cdot \right) }{\PP\left( \|\bX\| > x\right)} \longrightarrow_w r^{-\alpha} \sigma(\cdot) $$
as $x \to \infty$, where $\to_w$ denotes weak convergence. The limit measure 
$\sigma$ is called spectral measure.
\medskip

By Corollary 5.18 in \citet{resnick-08}, multivariate regular variation of $\bX$
with spectral measure $\sigma$ is equivalent to the fact that the distribution 
function $F$ of $\bX$ is in the max-domain of attraction of a multivariate 
extreme value distribution, i.e.
\begin{align*} %\label{eq:mda-multi}
F^n(a_n x_1 + b_n, \ldots, a_n x_d + b_n) \stackrel{n \to \infty}{\longrightarrow} G(x_1,\ldots,x_d),
\qquad x_1,\ldots,x_d > 0. 
\end{align*}
The limit distribution $G$ necessarily has $\Phi_\alpha$ marginal distributions
and is of the form 
\begin{align*}
G(x_1,\ldots,x_d) = \exp\left(- \mu\left\{ [0,\infty)^d \setminus ([0,x_1] \times \ldots \times [0,x_d] ) \right\} \right),
\end{align*}
for some Radon measure $\mu$ on $E = [0,\infty)^d \setminus \{\mathbf{0}\}$, 
the so-called \emph{exponent measure} $\mu$ of $G$. The exponent measure $\mu$
and the spectral measure $\sigma$ are related via
\begin{align*}
\mu(\{ \bx \in E: \, \|\bx\| > r, \, \bx / \|\bx\| \in A\}) = r^{-\alpha} \sigma(A), \qquad r>0, \ A \subset \mathbb{S}^{d-1}.
\end{align*}

The time series $X$ is called \emph{regularly varying} if all the finite-dimensional 
margins $(X_{t_1},\ldots,X_{t_d})$, $t_1 < t_2 < \ldots < t_d \in \ZZ$, $d \in \NN$,
are multivariate regularly varying. By \citet{basrak-segers-09}, regular variation
of $X$ is equivalent to the existence of a process $Y = (Y_t)_{t \in \ZZ}$ with
$\PP(Y_0 > y) = y^{-\alpha}$ for $y \geq 1$ such that, for every $s < t \in \ZZ$,
\begin{align} \label{eq:tail-proc}
\PP\bigg( \left(\frac{X_s}x,\ldots,\frac{X_t}{x}\right) \in \cdot \, \bigg| \, X_0 > x \bigg) 
\longrightarrow_d \PP\left( (Y_s, \ldots, Y_t) \in \cdot \right), \quad \text{as } x \to \infty. 
\end{align}
The process $Y$ is called \emph{tail process} of $X$, see \citet{basrak-segers-09} and \citet{planinic-soulier-18} for more details and further properties.
\medskip

In the following, we will always assume that the time series $X$ is regularly
varying with tail process $Y$. Furthermore, the probability measure induced by
the random vector $(Y_i)_{i \in I}$ will be called $\mu_I$ for any index set
$I \subset \ZZ$. 

\section{Distribution of Clusters of Extremes and Ordinal Patterns} \label{sec:clust-dist}

In this section, we will analyze the limiting behaviour of the size and ordinal pattern of $u$-exceedance clusters in the time series $X$ with $u$-exceedances 
being defined as in \eqref{eq:def-cluster}. Intuitively, the following expression gives a plausible definition of the distribution of the size $C_u$ of a randomly selected $u$-exceedance cluster:
\begin{align*}
 \PP(C_u = l) ={}& \lim_{n \to \infty} \frac{\#\{ u\text{-exceedance clusters of size } l
	\text{ in } (X_t)_{t=-1}^n\}}{\#\{u\text{-exceedance clusters (of any size) in } (X_t)_{t=-1}^n\}}\\
={}& \lim_{n \to \infty} \frac{\sum_{k=0}^{n-l} \mathbf{1}_{\{X_{k-1} \leq u, X_k > u, \ldots, X_{k+l-1} > u, X_{k+l} \leq u \}}}
{\sum_{l \in \NN} \sum_{k=0}^{n} \mathbf{1}_{\{X_{k-1} \leq u, X_k > u, \ldots, X_{k+l-1} > u, X_{k+l} \leq u \}}}\\
={}& \lim_{n \to \infty} \frac{\sum_{k=0}^{n-l} \mathbf{1}_{\{X_{k-1} \leq u, X_k > u, \ldots, X_{k+l-1} > u, X_{k+l} \leq u \}}}
{\sum_{k=0}^{n} \mathbf{1}_{\{X_{k-1} \leq u, X_k > u\}}}, \qquad l \in \NN.
\end{align*}
If $X$ is ergodic, we can apply the pointwise Birkhoff-Khinchin theorem
and obtain that the above limit almost surely exists and equals
\begin{equation} \label{eq:distr-u-exceed}
 \PP(C_u = l) = \frac{\PP(X_{-1} \leq u, X_0 > u, \ldots, X_{l-1}>u, X_l \leq u)}{\PP(X_{-1} \leq u, X_0 > u)}, \quad l \in \NN.
\end{equation} 
Dividing both the enumerator and the denominator of \eqref{eq:distr-u-exceed} by
$\PP(X_0>u)$, we can see from relation \eqref{eq:tail-proc} that the
distribution of $C_u$ eventually becomes independent from the threshold $u$ as 
$u \to \infty$ provided that 
$$\mu_{\{-1,\ldots,l\}}\{ (x_i)_{i=-1}^l \in [0,\infty)^{l+2}: \, x_i = 0 \text{ or } x_i=1 \text{ for some } i=-1,\ldots,l\} = 0$$
for every $l \in \NN$:
\begin{align}
\lim_{u \to \infty} \PP(C_u = l) ={}& \lim_{u \to \infty} \frac{\PP(X_{-1} \leq u, X_0 > u, \ldots, X_{l-1}>u, X_l \leq u \mid X_0 > u)}{\PP(X_{-1} \leq u, X_0 > u \mid X_0 > u)} \nonumber \\
={}& \frac{\PP( Y_{-1} \leq 1, Y_0 > 1,\ldots, Y_{l-1} > 1, Y_l \leq 1)}{\PP(Y_{-1} \leq 1, Y_0 > 1)} \nonumber \\
={}& \frac{\mu_{\{-1,\ldots,l\}}([0,1] \times (1,\infty)^l \times [0,1])}{\mu_{\{-1,0\}}([0,1] \times (1,\infty))},  \quad l \in \NN. \label{eq:limit-len}
\end{align}

\begin{example}
 We consider two examples to compare the limiting distribution of $C_u$ as $u \to \infty$, i.e.\ the distribution of the cluster size according to our definition, with the limiting cluster size distribution in the classical setting which has been studied extensively in the literature, cf.~\citet{robert-09} and references therein. While both distributions coincide in the first examples, they significantly differ in the second one.
 \begin{enumerate}
   \item We consider a first order max-autoregressive model \citep[cf.][]{davis-resnick-89}
   \begin{equation} \label{eq:mar}
     X_t = \max\{ a X_{t-1}, (1-a) Z_t\}, \quad t \in \ZZ, 
   \end{equation}
   where $a \in [0,1]$ and $\{Z_t\}_{t \in \ZZ}$ is a unit Fr\'echet noise
   process. Equation \eqref{eq:mar} possesses a stationary solution with unit Fr\'echet margins that is regularly varying. For $t > 0$, the tail process $\{Y_t\}_{t \in \ZZ}$ is given by  
   $$ (Y_0, Y_1,\ldots, Y_t) =_d (Y, aY, \ldots, a^t Y) $$
   where $Y$ is a standard Pareto random variable. Furthermore, we have
   \begin{equation*}
     Y_{-1} = \begin{cases} 
        a^{-1}Y & \quad \text{with probability }   a,\\
        0       & \quad \text{with probability } 1-a.
        \end{cases}
   \end{equation*}
   Thus, for $l \in \NN$, we obtain
   \begin{align}
   \lim_{u \to \infty} \PP(C_u = l) 
   ={}& \frac{\PP( Y_{-1} \leq 1, Y_0 > 1,\ldots, Y_{l-1} > 1, Y_l \leq 1)}{\PP(Y_{-1} \leq 1, Y_0 > 1)} \nonumber \\
   ={}& \frac{\PP(Y_{-1} = 0, a^{l-1} < Y < a^l)}{\PP(Y_{-1}=0)} = a^{l-1} (1-a), \label{eq:geo}
   \end{align}
   i.e.\ the limiting distribution is a geometric distribution with parameter $1-a$. Alternatively, this distribution could also be derived
   from Proposition 1 in \citet{markovich-17} plugging the formulae
   for $\PP(T_2(u) = l+1)$ into the expression $\lim_{u \to \infty} 
   \PP(T_2(u) = l+1 \mid T_2(u) > 1)$ and using that the extremal index
   is given by $\theta=1-a$. However, obtaining a closed-form
   expression seems to be much simpler using the tail process.   
   
   Note that the limiting geometric distribution in \eqref{eq:geo} coincides with the limiting distribution for the size of clusters defined in the classical sense,
   cf.\ \citet{perfekt94}. This is due to the fact that, in the limit,
   exceedances over high thresholds always occur subsequently.  
   \item As second example, we consider a stationary moving maximum process
   \citep[cf.][for instance]{deheuvels83} defined by
   $$ X_t = \frac 1 2 \max\{ Z_t, Z_{t-2}\}, \quad t \in \ZZ, $$
   with $\{Z_t\}_{t \in \ZZ}$ being a unit Fr\'echet noise process.  
   By definition, we have that $\PP(X_t \leq u \mid X_0 > u) \to 1$
   for all $t \in 2 \ZZ + 1$. Consequently,
   $$ \lim_{u \to \infty} \PP(C_u > 1) = \lim_{u \to \infty}
    \frac{\PP(X_{-1} \leq u, X_0 > u, X_1 > u)}{\PP(X_{-1} \leq u, X_0 > u)}
    = 0,$$
   i.e.\ the distribution of $C_u$ converges weakly to a Dirac measure in $1$. 
   
   In contrast, the limiting cluster size distribution according to the classical definition is obviously a Dirac measure in $2$, that is, exceedances over high thresholds always occur in pairs. As the exceedances are separated by a non-exceedance, each pair is considered as two single clusters in our definition, while they belong to the same cluster according to the classical definition.  
 \end{enumerate}
\end{example}

Similarly to their size, we can investigate ordinal patterns in $u$-exceedance clusters. Here, for fixed $l \in \NN$, we are interested in the distribution of the $l$-ordinal pattern of a (randomly selected) $u$-exceedance cluster of size $l$:
\begin{align*}
\PP_{u,l}(\pi) ={}& \PP( \Pi((X_i)_{i=0}^{l-1}) = \pi \mid (X_i)_{i=0}^{l-1} \text{ is a } u\text{-exceedance cluster}) \\
={}&
\frac{\PP(\Pi((X_i)_{i=0}^{l-1})=\pi, X_{-1} \leq u, X_0>u,\ldots, X_{l-1}>u, X_l \leq u)}
{\PP(X_{-1} \leq u, X_0>u,\ldots, X_{l-1}>u, X_l \leq u)} 
\end{align*}
for each $\pi \in S_{l-1}$, i.e.\ $\PP_{u,l}$ defines a probability distribution on $S_{l-1}$. Again, this distribution converges as $u \to \infty$:
\begin{align}
\lim_{u \to \infty} & \PP_{u,l}(\pi)
={} \frac{\PP(\Pi((Y_i)_{i=0}^{l-1})=\pi, Y_{-1} \leq 1, Y_0>1,\ldots, Y_{l-1}>1, Y_l \leq 1)}
{\PP( Y_{-1} \leq 1, Y_0>1,\ldots, Y_{l-1}>1, Y_l \leq 1)} \nonumber
\displaybreak[0]\\
%={}& \PP(\Pi((Y_i)_{i=0}^{l-1})=\pi \mid Y_{-1} \leq 1, Y_0>1,\ldots, Y_{l-1}>1, Y_l \leq 1) \\
={}& \frac{\mu_{\{-1,\ldots,l\}}(\{((x,\by,z) \in [0,1] \times (1,\infty)^l \times [0,1]:\ \Pi(\by)=\pi\})}
{\mu_{\{-1,\ldots,l\}}([0,1] \times (1,\infty)^l \times [0,1]))}.
\label{eq:limit-ord}
\end{align}

\section{Asymptotic Results for Empirical Estimators} \label{sec:asymptotics}

According to Equation \eqref{eq:limit-len} and Equation \eqref{eq:limit-ord}, both the limit distribution of clusters and the limit distribution of ordinal patterns within a cluster are given by a ratio of the type 
$$\mu_{\{-1,\ldots,t\}}(A) / \mu_{\{-1,\ldots,t_0\}}(A_0)$$
for appropriate dimension $t, t_0 \in \NN_0$ and sets $A \subset [0,\infty) \times (1,\infty) \times [0,\infty)^t$ and 
$A_0 \subset [0,\infty) \times (1,\infty) \times [0,\infty)^{t_0}$ with $\mu_{\{-1,\ldots,t_0\}}(A_0)>0$, i.e.\ a ratio of measures of two sets that are bounded from below by $1$ in their second component. 

More precisely, in case of the cluster size distribution in \eqref{eq:limit-len}, we have $t=l$, $A = [0,1] \times (1,\infty)^l \times [0,1]$, $t_0=0$ and $A_0 = [0,1] \times (1,\infty)$, while, in case of the distribution of the $l$-ordinal pattern in \eqref{eq:limit-ord}, we have
$t=t_0=l$, $A = \{ (x, \by,z) \in [0,1] \times (1,\infty)^l \times [0,1]: \ \Pi(\by) = \pi\}$ and $A_0 = [0,1] \times (1,\infty)^l \times [0,1]$.

In the following, we will consider the general class of ratios of the above 
type, including both the limits in \eqref{eq:limit-len} and \eqref{eq:limit-ord} as special cases, and discuss their estimation from observations $X_{-1},\ldots,X_n$. It is worth noting that, making use of relation
$\mu_{\{-1,\ldots,t\}}(A) = \mu_{\{-1,\ldots,t+1\}}(A \times [0,\infty))$,
we may replace $t$ and $t_0$ by the maximum of the two, i.e.\ without loss of generality, we may assume that $t=t_0$.

Analogously to the calculations in Section \ref{sec:clust-dist}, one can show that
$$ \lim_{u \to \infty} \frac{\PP( (X_i)_{i=-1}^{t} \in uA)}{\PP( (X_i)_{i=-1}^{t} \in uA_0)} = \frac{\mu_{\{-1,\ldots,t\}}(A)} {\mu_{\{-1,\ldots,t\}}(A_0)} $$
provided that $\mu_{\{-1,\ldots,t\}}(\partial A) = \mu_{\{-1,\ldots,t\}}(\partial A_0) = 0$. This limit relation gives reason to
set a high threshold $u$ and use a ratio estimator of the type
\begin{equation} \label{eq:ratio-est}
\widehat{R}_{n, u}(A, A_0) = \frac{\widehat{P}_{n,u}(A)}{\widehat{P}_{n,u}(A_0)}
\end{equation}
where
\begin{equation} \label{eq:single-est}
 \widehat{P}_{n,u}(A) = \frac 1 n \sum_{k=1}^{n-t} \mathbf{1}_{ \{(X_i)_{i=k-1}^{k+t} \in u A\}}
\end{equation}       
is the empirical counterpart of the probability $\PP( (X_i)_{i=-1}^{t} \in uA)$.                  

We will show asymptotic properties of the ratio estimator $\widehat{R}_{n,u_n}(A, A_0)$ for some appropriate sequence of thresholds $(u_n)_{n \in \NN}$ such that $u_n \to \infty$ and $n \PP(X_0 > u_n) \to \infty$ as $n \to \infty$.
To obtain consistency, we also need a mixing condition for the time series
$(X_t)_{t \in \ZZ}$ expressed in terms of the $\alpha$-mixing coefficients 
\begin{align} \label{eq:def-alpha}
\alpha_h ={}&  \sup_{A,B \in \calB(\RR^\NN)} |\PP((X_t)_{t \leq 0} \in A, \, (X_t)_{t \geq h} \in B) \nonumber \\
& \hspace{2cm} - \PP((X_t)_{t \leq 0} \in A) \PP((X_t)_{t \geq h} \in B)|, \qquad h \geq 0.
\end{align}
Here, we will assume that $(X_t)_{t \in \ZZ}$ is $\alpha$-mixing,
i.e.\ $\alpha_h \to 0$ as $h \to \infty$, and that the coefficients $(\alpha_h)_{h=0}^\infty$ are summable. The proof is postponed to the appendix.

\begin{proposition} \label{prop:consistency}
 Let $(X_t)_{t \in \ZZ}$ be a regularly varying, strictly stationary time series with tail process $(Y_t)_{t \in \ZZ}$ whose finite-dimensional distributions are given by
 $(\mu_I)_{I \subset \ZZ}$. Assume that the corresponding $\alpha$-mixing coefficients satisfy 
 \begin{equation} \label{eq:alpha-decay-consistency}
 \alpha_n \in \mathcal{O}(n^{-\delta}) \quad \text{for all } n \in \NN \text{ and some } \delta > 1.
 \end{equation}
 Moreover, let $A_0, A_1 \subset [0,\infty) \times (1,\infty) \times [0,\infty)^t$ be continuity sets w.r.t.\ $\mu_{\{-1,\ldots,t\}}$
 such that $\mu_{\{-1,\ldots,t\}}(A_0) >  0$.
 Then, for any sequence $(u_n)_{n \in \NN}$ satisfying $u_n \to \infty$
 and 
 \begin{equation*}
  n^{\delta/(1+\delta)} \PP(X_0 > u_n) \to \infty,
 \end{equation*}
 we have
 $$ \widehat{R}_{n, u_n}(A_1, A_0) \to_p \frac{\mu_{\{-1,\ldots,t\}}(A_1)}{\mu_{\{-1,\ldots,t\}}(A_0)}. $$	
\end{proposition}
\medskip

To further establish asymptotic normality, we will assume the following mixing condition which combines a condition on the $\alpha$-mixing coefficients and an anti-clustering condition.
 
\begin{condition}[M]
	There exist a sequence $\{u_n\}_{n \in \NN} \subset \RR$ of thresholds and 
	an intermediate sequence
	$\{r_n\}_{n \in \NN} \subset \NN$ 
	with $\lim_{n \to \infty} u_n = \lim_{n \to \infty} r_n = \infty$,
	$ \lim_{n \to \infty} n \PP(X_0 > u_n) = \infty$, $\lim_{n \to \infty} 
	r_n \PP(X_0 > u_n) = 0$ such that
	\begin{equation} \label{eq:mixing}
	\lim_{n \to \infty} \frac 1 {\PP(X_0 > u_n)} \sum\nolimits_{h=r_n}^\infty \alpha_h = 0
	\end{equation}
	and
	\begin{equation} \label{eq:anticlustering}
	\lim_{k \to \infty} \limsup_{n \to \infty} \sum\nolimits_{h=k}^{r_n} \PP(X_h > u_n \mid X_0 > u_n) = 0.
	\end{equation}
\end{condition}

\begin{remark}
Condition (M) is an adapted version of the condition in \citet{davis-mikosch-09} who consider a ratio estimator of a similar type
for the so-called extremogram. Similarly to the condition for consistency given in Equation \eqref{eq:alpha-decay-consistency}, also the mixing condition in Equation \eqref{eq:mixing} implies conditions on the decay of the sequence $\{\alpha_h\}_{h \in \NN}$ which will be discussed below in more detail.
	
The anti-clustering condition in Equation \eqref{eq:anticlustering} is very similar to Condition (2.8) in \citet{davis-hsing-95} and Condition 4.1 in \citet{basrak-segers-09}. By Proposition 4.2 in \citet{basrak-segers-09}, it implies that the tail process $\{Y_t\}_{t \in \ZZ}$ converges to $0$ almost surely as $|t| \to \infty$ and thus ensures finite cluster size.
\end{remark}

To prove the asymptotic normality of the ratio estimators, we first make use of
the following auxiliary result in \citet{davis-mikosch-09}, adapted to our setting.

\begin{lemma}{\citep[][Thm.~3.1]{davis-mikosch-09}} \label{lem:variance}
	Let $(X_t)_{t \in \ZZ}$ be a regularly varying, strictly stationary time series 
	with tail process $(Y_t)_{t \in \ZZ}$ whose finite-dimensional distributions 
	are given by $(\mu_I)_{I \subset \ZZ}$. Moreover, let 
	$A \subset [0,\infty) \times (1,\infty) \times [0,\infty)^t$ be a continuity 
	set w.r.t.~$\mu_{\{-1,\ldots,t\}}$. If Condition (M) holds, then
	\begin{align*}
	& \lim_{n \to \infty} n \PP(X_0 > u_n) \Var\left(\frac{\widehat{P}_{n,u_n}(A)}{\PP(X_0 > u_n)}\right) \\
	={}& \mu_{\{-1,\ldots,t\}}(A) 
	+ 2 \sum\nolimits_{h=1}^\infty \PP((Y_i)_{i=-1}^t \in A, \, (Y_i)_{i=h-1}^{h+t} \in A) < \infty.
	\end{align*}
\end{lemma}

We further proceed by noting that Equation \eqref{eq:mixing} and $r_n \PP(X > u_n) \to 0$ imply that
$$ \lim_{n \to \infty} r_n \sum\nolimits_{h=r_n}^\infty \alpha_h = 0.$$
Thus,
$$\liminf_{n \to \infty} n \sum\nolimits_{h=n}^\infty \alpha_h = 0$$
and, consequently, $\liminf_{n \to \infty} n^2 \alpha_n = 0$. Imposing the existence
of a finite limit superior, we may conclude that there exists some $\delta \geq 0$ such that
\begin{equation} \label{eq:alpha-decay}
\alpha_n \in \mathcal{O}(n^{-\delta}) \quad \text{for some }\delta \geq 2.
\end{equation}
Using this slight strengthening of Condition (M), we can verify asymptotic normality of the estimators $\widehat{P}_{n,u_n}/\PP(X_0 > u_n)$. The proof is postponed to the appendix.

\begin{theorem} \label{thm:asymptotic-normality}
	Let $(X_t)_{t \in \ZZ}$ be a regularly varying, strictly stationary time 
	series with tail process $(Y_t)_{t \in \ZZ}$ whose finite-dimensional distributions are 
	denoted by $(\mu_I)_{I \subset \ZZ}$. Moreover, let $A_0, \ldots, A_N \subset 
	[0,\infty) \times (1,\infty) \times [0,\infty)^t$ be continuity sets w.r.t.\
	$\mu_{\{-1,\ldots,t\}}$. We further assume that Condition (M) and 
	\eqref{eq:alpha-decay} hold and 
	\begin{equation} \label{eq:add-cond}
	\lim_{n \to \infty} n^{\delta/(4+\delta)} \PP(X_0 > u_n) = \infty,
	\end{equation}
	if $\delta >2$ in \eqref{eq:alpha-decay} or
	\begin{equation} \label{eq:add-cond-delta0}
	\lim_{n \to \infty} n^{1/2} \frac{\PP(X_0 > u_n)^{3/2}}{|\log(\PP(X_0 > u_n))|} = \infty
	\end{equation}
	if $\delta=2$. Then,
	\begin{align*}
	& \sqrt{n \PP(X_0 > u_n)} \left( \begin{array}{c} \frac{\widehat{P}_{n,u_n}(A_0)}{\PP(X_0 > u_u)} - \PP((X_i)_{i=-1}^t \in u_n A_0 \mid X_0 > u_n) \\ 
	\vdots \\ \frac{\widehat{P}_{n,u_n}(A_m)}{\PP(X_0 > u_n)} - \PP((X_i)_{i=-1}^t \in u_n A_N \mid X_0 > u_n) \end{array} \right) \\
	& \hspace{9cm} \longrightarrow_d  \mathcal{N}(\mathbf{0}, \Sigma),
	\end{align*}
	as $n \to \infty$, where $\Sigma = (\sigma_{jl})_{0 \leq j,l \leq N}$ with
	\begin{align*}
	 \sigma_{j,l} ={}& \mu_{\{-1,\ldots,t\}}(A_j \cap A_l) + \sum\nolimits_{h=1}^\infty \PP((Y_i)_{i=-1}^t \in A_j, \ (Y_i)_{i=h-1}^{h+t} \in A_l) \\
	& \hspace{2.9cm} + \sum\nolimits_{h=1}^\infty \PP((Y_i)_{i=-1}^t \in A_l, \ (Y_i)_{i=h-1}^{h+t} \in A_j).
	\end{align*}
\end{theorem}

\begin{remark}
By regular variation,
\begin{align*}
   \frac{\EE\{\widehat{P}_{n,u_n}(A_j)\}}{\PP(X_0 > u_n)}
={}& \PP((X_i)_{i=-1}^t \in u_n A_j \mid X_0 > u_n)\\
\to{}& 
\PP((Y_i)_{i=-1}^t \in A_j) = \mu_{\{-1,\ldots,t\}}(A_j) \quad (n \to \infty),
\end{align*}
i.e.\ the conditional distribution is approximately normal around the
desired value even though the bias might be not negligible asymptotically.
If the limit expression vanishes, i.e.\ if we have $\mu_{\{-1,\ldots,t\}}(A_j) = 0$, we obtain the asymptotic
variance $\sigma_{jj}=0$, i.e.\ the limit distribution is degenerate. 
\end{remark}	
\medskip

Using the same arguments as in the proof of Corollary 3.3 in \citet{davis-mikosch-09},
we obtain the following corollary.

\begin{corollary} \label{coro:asymptotic-normality}
	Under the assumptions of Theorem \ref{thm:asymptotic-normality}, we have that
	\begin{align*}
	& \sqrt{n \PP(X_0 > u_n)} \left( \begin{array}{c} \widehat{R}_{n,u_n}(A_1,A_0) - \frac{\PP((X_i)_{i=-1}^t \in u_n A_1)}{\PP((X_i)_{i=-1}^t \in u_n A_0)} \\ 
	\vdots \\ \widehat{R}_{n,u_n}(A_N,A_0) - \frac{\PP((X_i)_{i=-1}^t \in u_n A_N)}{\PP((X_i)_{i=-1}^t \in u_n A_0)} \end{array} \right) \\
	& \hspace{6cm} \longrightarrow_d \mathcal{N}(\mathbf{0}, \mu_{\{-1,\ldots,t\}}(A_0)^{-4} F \Sigma F^\top),  
	\end{align*}
	as $n \to \infty$, where 
	$$ F = \begin{pmatrix} 
	 - \mu_{\{-1,\ldots,t\}}(A_1) &
	\mu_{\{-1,\ldots,t\}}(A_0) & 0 & 0 & \ldots & 0 \\
	- \mu_{\{-1,\ldots,t\}}(A_2) & 0 & \mu_{\{-1,\ldots,t\}}(A_0) & 0 & \ldots & 0 \\
	\vdots & \vdots & \vdots & \ddots & \vdots & \vdots \\
	- \mu_{\{-1,\ldots,t\}}(A_N) & 0 & 0 & 0 & \ldots & \mu_{\{-1,\ldots,t\}}(A_0)
	\end{pmatrix} $$
	and $\Sigma$ is given in Theorem \ref{thm:asymptotic-normality}. \\       
	If, in addition,
	\begin{equation*}
	\lim_{n \to \infty} \sqrt{n\PP(X_0 > u_n)} \left[ \frac{\PP((X_i)_{i=-1}^t \in u_n A_i)}{\PP((X_i)_{i=-1}^t \in u_n A_0)} -  \frac{\mu_{\{-1,\ldots,t\}}(A_i)}{\mu_{\{-1,\ldots,t\}}(A_0)} \right] = 0,
	\end{equation*}
	for $i=1,\ldots,N$, then   
	\begin{align*}
	& \sqrt{n \PP(X_0 > u_n)} \left( \begin{array}{c} \widehat{R}_{n,u_n}(A_1,A_0) - \frac{\mu_{\{-1,\ldots,t\}}(A_1)}{\mu_{\{-1,\ldots,t\}}(A_0)} \\ 
	\vdots \\ \widehat{R}_{n,u_n}(A_N,A_0) - \frac{\mu_{\{-1,\ldots,t\}}(A_N)}{\mu_{\{-1,\ldots,t\}}(A_0)} \end{array} \right) \\
	& \hspace{6.5cm} \to \mathcal{N}(\mathbf{0}, \mu_{\{-1,\ldots,t\}}(A_0)^{-4} F \Sigma F^\top). 
	\end{align*}
\end{corollary}

\begin{remark}
  Alternatively to our proofs, we could also show asymptotic normality 
  of the vectors $(\widehat{P}_{n,u_n}(A_j))_{j=0}^N$ and $(\widehat{R}_{n,u_n}(A_j))_{j=1}^N$, respectively, using slightly adapted versions
  of Theorem 3.2 and Corollary 3.3 in \citet{davis-mikosch-09}.
  Therein, besides Condition (M), they also assume the conditions
  \begin{equation} \label{eq:alpha-r}
  \lim\nolimits_{n \to \infty} n \PP(X_0 > u_n) \cdot \alpha_{r_n} = 0
  \end{equation}
  and 
  \begin{equation} \label{eq:add-cond-mk}
  \lim\nolimits_{n \to \infty} n^{1/3} \PP(X_0 > u_n) = \infty.
  \end{equation}
  By using different techniques in the proof and extending Condition (M) by the slightly stronger assumption \eqref{eq:alpha-decay}, we are able to
  drop condition \eqref{eq:alpha-r}. Furthermore, we replace condition \eqref{eq:add-cond-mk} by conditions \eqref{eq:add-cond} and \eqref{eq:add-cond-delta0}, respectively.
  
  For $\delta >2$, condition \eqref{eq:add-cond} is weaker than condition 
  \eqref{eq:add-cond-mk}, which is the limiting case of condition \eqref{eq:add-cond} as $\delta \searrow 2$. If $\alpha_h$ even decays exponentially, i.e.\ if condition \eqref{eq:add-cond} holds for $\delta$
  being arbitrarily large, the condition
  simplifies to $\lim_{n \to \infty} n^{1-\varepsilon} \PP(X_0 > u_n) = \infty$ for some $\varepsilon$ which is close to the minimal assumption
  $\lim_{n \to \infty} n \PP(X_0 > u_n) = \infty$
  stated in Condition (M).  
  
  For $\delta=2$, condition \eqref{eq:add-cond-delta0} is slightly stronger than condition \eqref{eq:add-cond-mk}.	However, it is still weaker than
$$\lim\nolimits_{n \to \infty} n^{1/3 - \varepsilon} \PP(X_0 > u_n) = \infty$$
  for $\varepsilon > 0$.
  
  Thus, even though our assumptions are not necessarily weaker than the assumptions in \citet{davis-mikosch-09} due to the fact that we further
  assume \eqref{eq:alpha-decay}, our results allow for a simplification 
  of Equations \eqref{eq:alpha-r} and \eqref{eq:add-cond-mk}.
\end{remark}
\medskip

In practical applications, the usability of central limit theorems in the flavor 
of Corollary \ref{coro:asymptotic-normality} for uncertainty assessment of the
resulting estimates is often limited by two obstacles:
\begin{itemize}
	\item The rate of convergence includes the unknown threshold exceedance probability
	$\PP(X_0 > u_n)$.
	\item The asymptotic (co-)variances are complex expressions including series expressions as given in Theorem \ref{thm:asymptotic-normality}.	          
\end{itemize} 
In the case of Corollary \ref{coro:asymptotic-normality}, however, one can cope with both difficulties in the following way:
\begin{itemize}
	\item Applying Lemma \ref{lem:variance} to the set $A = [0,\infty) \times (1,\infty)$, we obtain that, under Condition (M), $$\frac{\widehat{P}_{n,u_n}([0,\infty) \times (1,\infty))}{\PP(X_0 > u_n)} = 
	\frac{1}{n \PP(X_0 > u_n)} \sum\nolimits_{k=1}^n \mathbf{1}_{\{X_k > u_n\}} \to_p 1.$$
	Therefore, Corollary \ref{coro:asymptotic-normality} stills holds true
	if we replace the exceedance probability 
	$\PP(X_0 > u_n)$ by its empirical
	counterpart $n^{-1} \sum_{k=1}^n \mathbf{1}_{\{X_k > u_n\}}$.
	\item Similarly to the asymptotic (co-)variances for the extremogram
	estimators, the asymptotic (co-)variances arising in Corollary \ref{coro:asymptotic-normality} can be estimated via various bootstrap techniques such as a stationary bootstrap \citep{davis-mikosch-cribben-12} or a multiplier block bootstrap
	\citep{drees-15}. In Section \ref{sec:simu-study}, we will make use 
	of the multiplier block bootstrap which has been demonstrated to provide more accurate and robust results than the stationary bootstrap 
	in a simulation study in \citet{DDSW18}.
\end{itemize}

\section{Example: Max-Stable Time Series} 
\label{sec:simu-study}

    An important example of a stationary regularly varying times series
    $(X_t)_{t \in \ZZ}$ is a stationary max-stable time series with 
	$\alpha$-Fr\'echet margins. According to \citet{dehaan-84}, such a
	time series can be represented as
	\begin{equation} \label{eq:spectral-maxstable}
	X_t = \bigvee\nolimits_{j=1}^\infty \left(\Gamma_j^{-1} W_t^{(j)}\right)^{1/\alpha}, \qquad t \in \ZZ, 
	\end{equation}
	where $\{\Gamma_j\}_{j \in \NN}$ denote the arrival times of a unit rate
	Poisson process and $(W_t^{(j)})_{t \in \ZZ}$, $j \in \NN$, are independent copies of
	a nonnegative time series $(W_t)_{t \in \ZZ}$ such that $\EE W_t = 1$ for all
	$t \in \ZZ$. Then, $(X_t)_{t \in \ZZ}$ is regularly varying with index $\alpha$ and its tail process $(Y_t)_{t \in \ZZ}$ is of the form
	\begin{equation} \label{eq:tail-maxstable}
	Y_t = P \cdot \widetilde{W_t}^{1/\alpha}, \qquad t \in \ZZ,
	\end{equation}
	where $P$ is an $\alpha$-Pareto random variable, i.e.\ $\PP(P>x) = x^{-\alpha}$,
	$x>1$, and $(\widetilde W_t)_{t \in \ZZ}$ is an independent time series with
	$$ \PP\left(\widetilde W \in A\right) = \int_{[0,\infty)^{\ZZ}} w(0) \mathbf{1}_{w /w(0) \in A} \PP(W \in \mathrm{d}w), \quad A \subset [0,\infty)^{\ZZ},$$
	see \citet{dombry-ribatet-15} for more details. 
    \medskip
	
	The dependence structure of a max-stable time series is often summarized by 
	its extremal coefficient function, that is, a sequence $\{\theta(h)\}_{h \in \ZZ} \subset [1,2]$
	given by the relation
	$$ \PP(X_h \leq x, X_0 \leq x) = \PP(X_0 \leq x)^{\theta(h)}, \qquad x > 0, \ h \in \ZZ.$$
	In particular, $X_h=X_0$ a.s.\ iff $\theta(h)=1$ and $X_h$ and $X_0$ are
	(asymptotically) independent iff $\theta(h)=2$.
	
	The extremal coefficient function can be used to provide sufficient conditions for condition (M). The proof is postponed to the appendix.
	
	\begin{proposition} \label{prop:maxstable}
	  Let $(X_t)_{t \in \ZZ}$ be a strictly stationary max-stable time series with $\alpha$-Fr\'echet margins and extremal coefficient function
	  $\{\theta(h)\}_{h \in \ZZ}$ such that $h \mapsto \theta(h)$ is 
	  non-increasing on $\NN_0$. If there exist a sequence $\{u_n\}_{n \in \NN} \subset \RR$ of thresholds and an intermediate sequence
	  $\{r_n\}_{n \in \NN} \subset \NN$ 
	  with $\lim_{n \to \infty} u_n = \lim_{n \to \infty} r_n = \infty$,
	  $ \lim_{n \to \infty} n \PP(X_0 > u_n) = \infty$, $\lim_{n \to \infty} 
	  r_n \PP(X_0 > u_n) = 0$ such that
	   \begin{equation} \label{eq:mixing-maxstable}
	   \lim_{n \to \infty} u_n^{\alpha} \sum\nolimits_{h=1}^\infty h^2 \left(2 - \theta(h+r_n)\right) = 0
	  \end{equation}	
	  and 	
	  \begin{equation} \label{eq:anticlustering-maxstable}
	  \lim_{k \to \infty} \sum\nolimits_{h=k}^{r_n} 2 - \theta(h) = 0,
	  \end{equation}
	  then Condition (M) holds.
	\end{proposition}

	The additional assumption in the second part of the corollary that ensures the
	asymptotic unbiasedness of the estimator cannot be verified in this general 
	setting. However, for the closely related extremogram \citep{davis-mikosch-09},
	we have, from Equation \eqref{eq:bounds-extremogram} that
	$$ \left| \PP(X_h > u_n \mid X_0 > u_n)- (2-\theta(h)) \right| \in \mathcal{O}(u_n^{-\alpha}) $$ 
	\citep[see also][Lemma A.1]{buhl-klueppelberg-18}, i.e.\ 
	$$ \sqrt{n \PP(X(0) > u_n)} \left| \PP(X_h > u_n \mid X_0 > u_n)- (2-\theta(h)) \right| \to 0 $$ 
	holds if and only if $n u_n^{-3\alpha} \to 0$. A similar behaviour might be expected for the conditional probabilities in Corollary \ref{coro:asymptotic-normality} which are of the same type. 
\medskip

In the following simulation study, we will focus on one of the most popular models for max-stable processes, namely the Brown--Resnick process, that is, a stationary max-stable time series
with unit Fr\'echet margins and processes $(W_t)_{t \in \ZZ}$ and 
$(\widetilde W_t)_{t \in \ZZ}$ in \eqref{eq:spectral-maxstable} and 
\eqref{eq:tail-maxstable}, respectively, of the form
$$ W_t = \widetilde W_t = \exp\left(G_t - \Var(G_t)/2\right), \quad t \in \ZZ,$$
for a centered Gaussian time series $(G_t)_{t \in \ZZ}$ with $G_0=0$ a.s.\ and
stationary increments. Thus, $(X_t)_{t \in \ZZ}$ is a stationary max-stable
process and its law is uniquely determined by the semi-variogram
$$ \textstyle \gamma(h) = \frac 1 2 \Var(G_{h} - G_0), \quad h \in \ZZ,$$
see \citet{kabluchko-etal-09}. In applications, often Brown--Resnick processes associated to a semi-variogram of power type
$$ \gamma(h) = C |h|^\beta, \quad h \in \ZZ,$$
for some $C > 0$ and $\beta \in (0,2]$ are considered.

Similarly to \citet{cho-etal-16, buhl-klueppelberg-18, buhl-etal-19, buhl-klueppelberg-19} for spatial and spatio-temporal Brown-Resnick processes, we can verify that such a Brown--Resnick process satisfies the assumptions of our central limit theorems (Theorem \ref{thm:asymptotic-normality} and Corollary \ref{coro:asymptotic-normality}, respectively). As our assumptions are different from the ones in the papers mentioned above, we will verify them independently making use of our results for general max-stable processes in Proposition \ref{prop:maxstable}. The proof is postponed to the appendix. 

\begin{corollary} \label{coro:br}
  Let $(X_t)_{t \in \ZZ}$ a Brown--Resnick time series associated to a semi-variogram $\gamma$ satisfying $\gamma(h) \geq C |h|^\varepsilon$ 
  for all $h \in \ZZ$ and some $C, \varepsilon > 0$. Then, choosing
  $u_n \sim n^{\beta_1}$ for some $\beta_1 \in (0,1)$ and $r_n \sim n^{\beta_2}$ for some $\beta_2 \in (0,\beta_1)$, the assumptions
  of Theorem \ref{thm:asymptotic-normality} and, thus, of the first part of Corollary \ref{coro:asymptotic-normality}, are satisfied.
\end{corollary}

\begin{figure} 
	\includegraphics[width=0.99\textwidth]{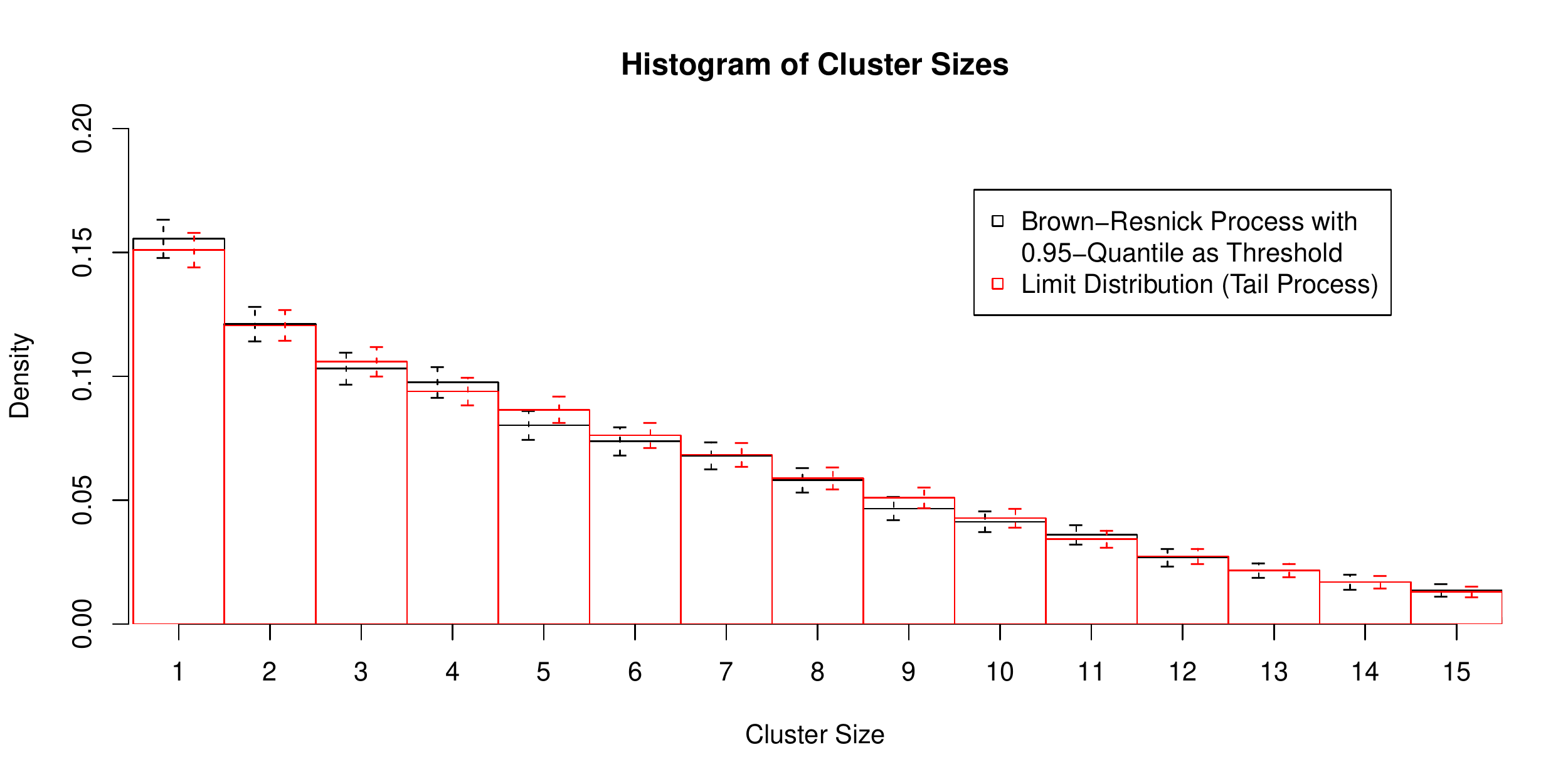}
	\caption{Histogram of the cluster size distribution for estimates from a simulated Brown--Resnick process (black) and the corresponding limit distribution (red). $95\,\%$ confidence intervals are added, obtained via a multiplier block bootstrap (black) and from the theoretical asymptotic distribution (red), respectively.} \label{fig:clustersize}
\end{figure}

\begin{figure}
	\includegraphics[width=0.99\textwidth]{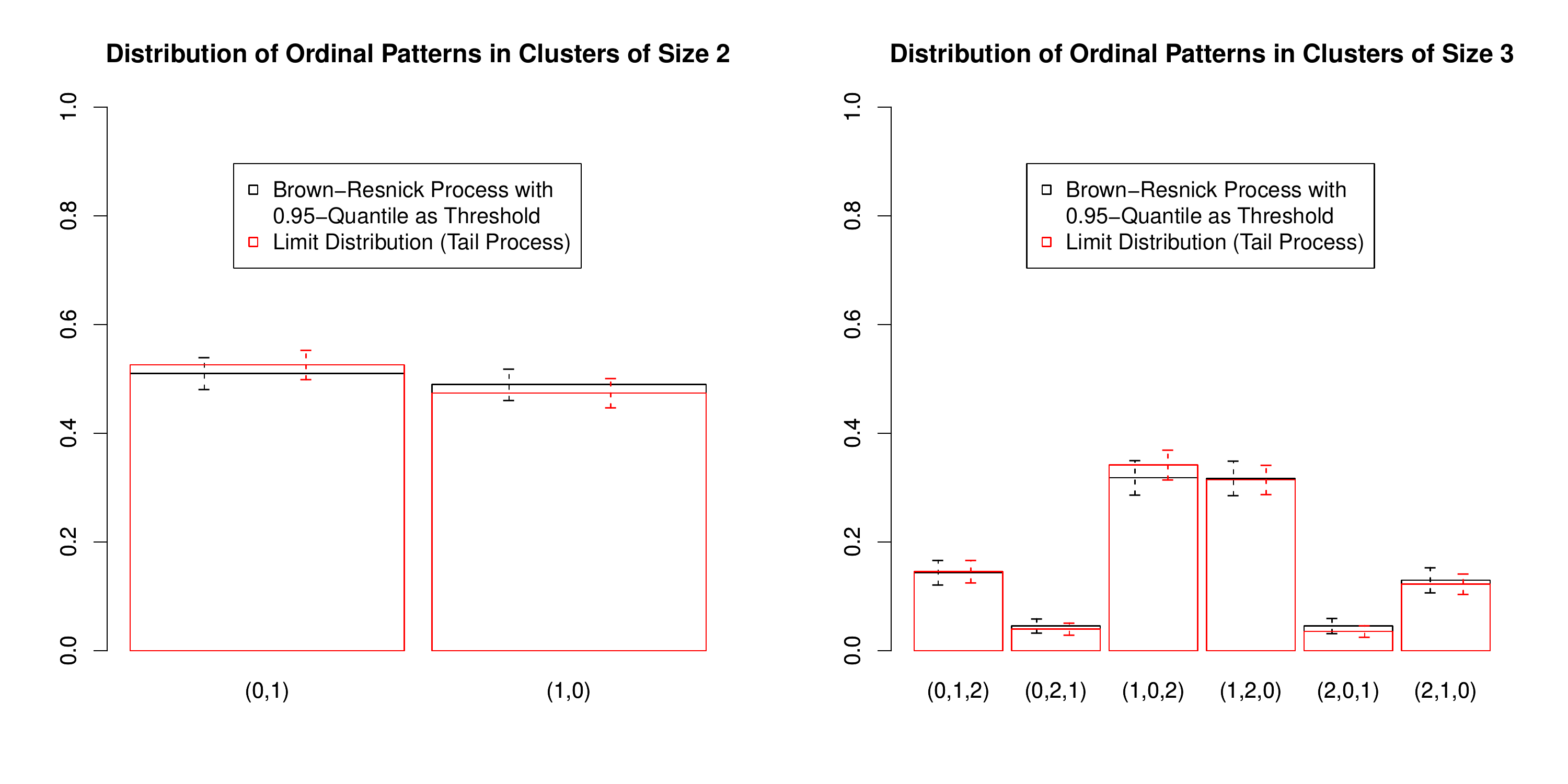}
	\caption{Bar charts of the distributions of $2$- (left) and $3$-ordinal patterns (right) within exceedance clusters estimated from a simulated Brown--Resnick process (black) and the corresponding limit distributions (red). The corresponding $95\,\%$ confidence intervals are obtained via a multiplier block bootstrap (black) and from the theoretical asymptotic distribution (red), respectively.}  \label{fig:ordinalpatterns}
\end{figure}

We now simulate a Brown--Resnick process to demonstrate the performance of the estimators of the type $\widehat{R}_{n,u}(\cdot,\cdot)$ for the distribution of the cluster size and the
ordinal patterns within a cluster. More precisely, we use the extremal functions approach \citep{dombry-etal-16} to simulate a Brown-Resnick time series of length $1\,000\,000$ with unit Fr\'echet margins and associated to the semi-variogram $\gamma(h) = 0.1 \cdot |h|^{1.75}$. We then estimate
\begin{itemize}
	\item the distribution of the cluster size,
	\item the distribution of ordinal patterns within clusters of size $2$
	\item and the distribution of ordinal patterns within clusters of size $3$
\end{itemize}
based on exceedances over the $95\%$-quantile of the unit Fr\'echet distribution using the ratio estimators according to Equation \eqref{eq:ratio-est}.
The results are displayed and compared to the exact limit distributions, calculated via simulations from the tail process, in Figures
\ref{fig:clustersize} and \ref{fig:ordinalpatterns}, respectively.
The uncertainty of the estimators is assessed via the multiplier block bootstrap \citep{drees-15,DDSW18} based on fixed blocks of size $1\,000$.
The $95\,\%$ confidence intervals obtained from the bootstrap are compared to the theoretical confidence intervals according to the asymptotic distribution given in Corollary \ref{coro:asymptotic-normality}.
It can be seen that all the probabilities are estimated rather accurately
and that the estimated uncertainty is close to the theoretical one, i.e.\ both types of confidence intervals have similar sizes.
\medskip

\section{Application: River Discharge at Cologne} \label{sec:app}

\begin{figure}[h]
	\includegraphics[width=0.99\textwidth]{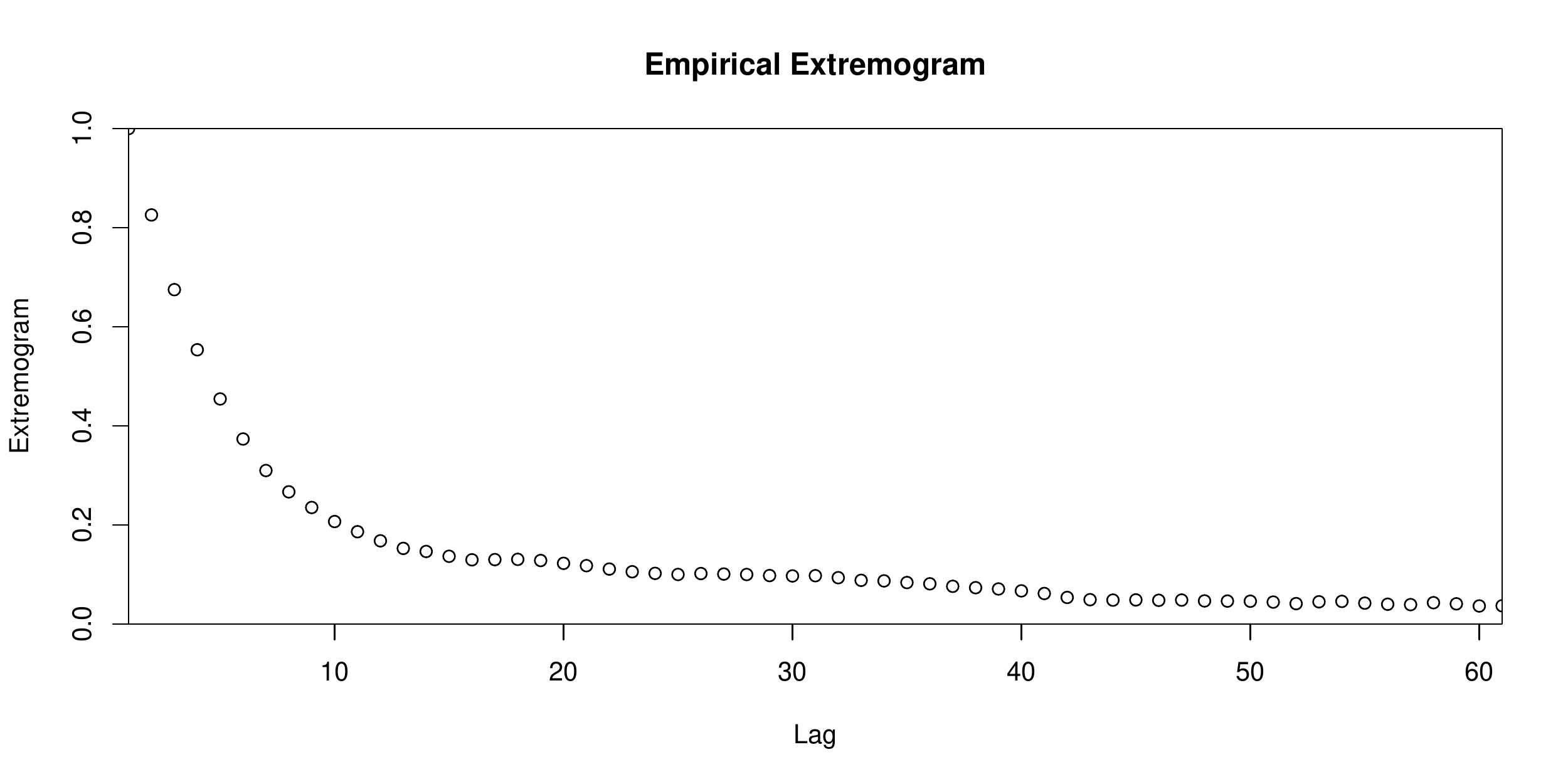}
	\caption{Empirical extremogram of the daily river discharge at Cologne based on the empirical $95\,\%$--quantile.}	
	\label{fig:extremogram_runoff}
\end{figure}

\begin{figure}[h] 
	\includegraphics[width=0.99\textwidth]{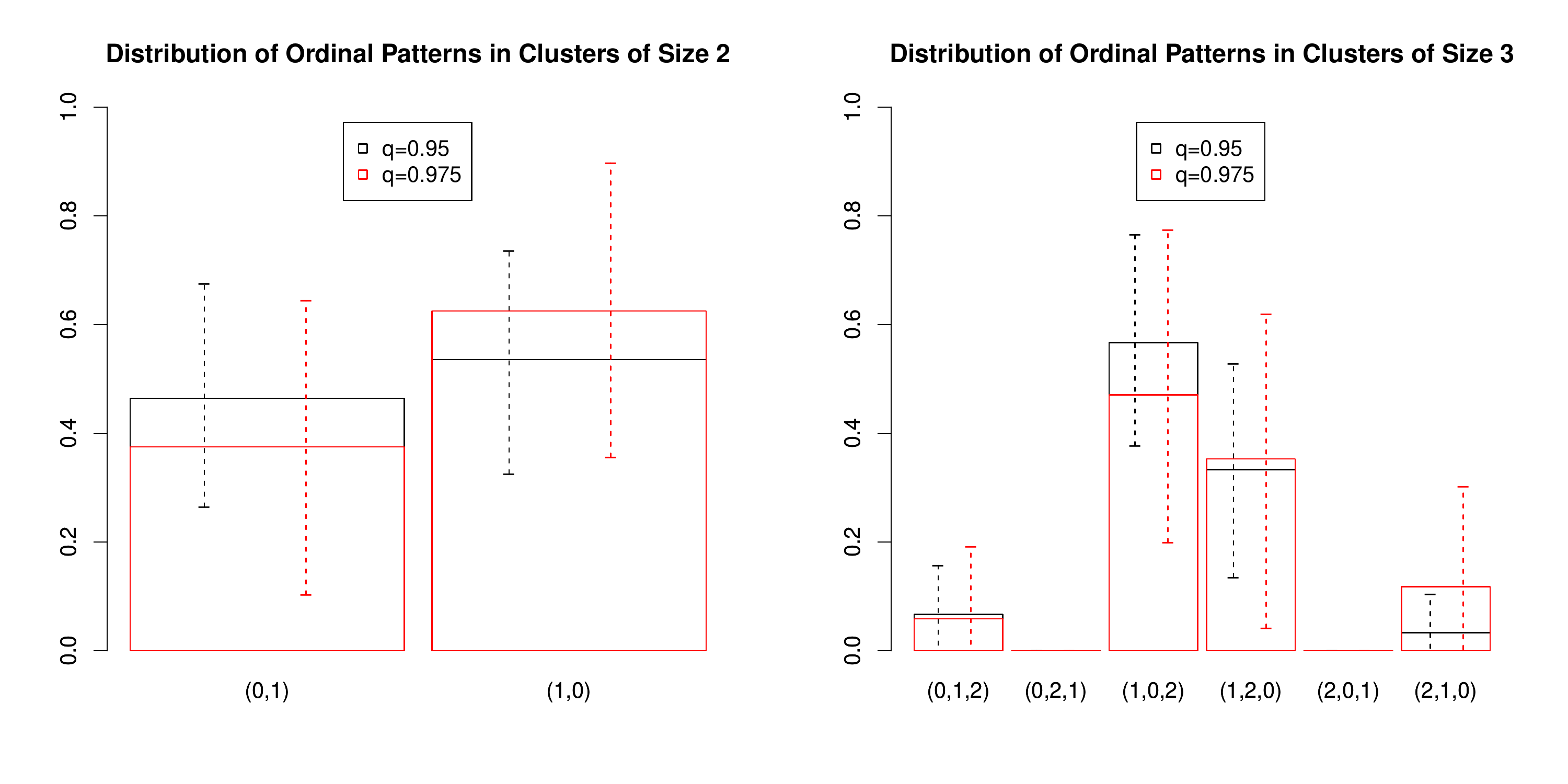}
	\caption{Bar charts of the distributions of ordinal patterns for the first $2$ (left) and $3$ (right) values of exceedance clusters of the daily river discharge at Cologne for three different thresholds (the empirical
		$95\,\%$-, $96\,\%$- and $97\,\%$-quantiles, from top to bottom).}
	\label{fig:ordinalpatterns-runoff}
\end{figure}

As an application we consider a time series of daily discharge data of the 
river Rhine measured at Cologne. In many cases, river discharge data exhibit 
temporal clustering of extremes, which entails the use of declustering 
techniques for the statistical analysis of their tail behaviour \citep[cf.][for
instance]{kallache-etal-11,asadi-etal-15}. Here, we study the structure of 
these clusters making use of the estimators introduced above. We restrict 
ourselves to the analysis of floods in the extended winter season (DJFM), assuming stationarity of the time series within each winter period consisting of 121 days
(and 122 days, respectively, in leak years). 

The given data set, provided by The Global Runoff Data Centre, 56068 Koblenz, Germany, consists of data from 197 winter seasons from December 1816 to March 2013. In an exploratory analysis, we calculate the empirical
version of the extremogram 
$$\rho_{(1,\infty),(1,\infty)}(h) 
 = \lim_{u \to \infty} \PP(X_h > u \mid X_0 > u), \quad h \in \ZZ,$$
based on exceedances the empirical $95\,\%$-quantile according to \citet{davis-mikosch-09}. The result, displayed in Figure \ref{fig:extremogram_runoff}, shows a decrease of extremal dependence as the temporal lag increases
being close to asymptotic independence for lags larger than $40$ days.  
Further analyses indicate that runoff data from different seasons may be assumed to be independent. These observations offer the applicability of the ratio estimator and the results on its asymptotic behaviour from Section \ref{sec:asymptotics}.

We choose two different thresholds for the empirical verification of the
stability of different exceedance cluster characteristics. More precisely,
we consider the empirical $95\,\%$- and $97.5\,\%$-quantiles
as thresholds leading to $200$ and $114$ clusters, respectively. 
As the empirical distributions of cluster sizes are rather difficult to compare due to the large number of potential outcomes relatively to the small number of clusters, we focus on the distribution of the $2$- and $3$-ordinal patterns. The results are displayed in Figure \ref{fig:ordinalpatterns-runoff} supplemented by $95\,\%$ confidence intervals obtained via a multiplier block bootstrap using each season as a fixed block. Even though the number of clusters is quite small, some interesting observations can be made: While both potential patterns of length two occur with almost the same frequency (in particular in case of the $95\,\%$-quantile), for clusters of size $3$, the patterns for which the second observation is the largest, i.e.\ $(1,0,2)$
and $(1,2,0)$, are clearly predominant. This means that extreme events that occur three time instants tend to show an ``up-down'' pattern. In contrast, patters with a ``down-up'' behaviour, i.e $(0,2,1)$ and $(2,0,1)$ do not occur at all.

In order to obtain more stable results based on a larger number of clusters, one could focus on patterns at the beginning of potentially longer clusters, i.e.\ one could consider the pattern for the first two instants within all clusters that are at least of size $2$, for instance. Maybe one can even use these ordinal patterns at the beginning of clusters to predict the length of the clusters. Such an analysis, however, is beyond the scope of the present article.

\section*{Acknowledgements}

Financial support by the DFG (German Research Foundation) for the project ``Ordinal-Pattern-Dependence: Grenzwerts\"atze und Strukturbr\"uche im langzeit- abh\"angigen Fall mit Anwendungen in Hydrologie, Medizin und Finanzmathematik'' (SCHN 1231/3-1) is gratefully acknowledged. In addition,
we would like to thank Dr.~Svenja Fischer (University of Bochum) for 
communicating the river discharge data set to us.
We are also grateful to the editor, an associate and two anonymous referees for valuable comments that helped to significantly improve the 
paper.

\section*{Appendix: Proofs}

\begin{proof}[Proof of Proposition \ref{prop:consistency}]
	We first note that, by Equation \eqref{eq:alpha-decay-consistency}, there exists
	some $C>0$ such that $\alpha_h \leq C n^{-\delta}$ for all $n \in \NN$.	
	By a straightforward computation, for $A = A_0$ and $A=A_1$ and any intermediate sequence $(r_n)_{n \in \NN}$ with $r_n \to \infty$ and 
	$r_n/n \to 0$, we then obtain 
	\begin{align*}
	& \Var\left(\frac{\widehat{P}_{n,u_n}(A)}{\PP(X_0>u_n)}\right) \\
	={}& \frac{1}{n^2 \PP(X_0>u_n)^2} \sum_{k=1}^{n-t} \sum_{l=1}^{n-t}
	\Cov\left(\mathbf{1}_{\{(X_i)_{i=k-1}^{k+t} \in u_nA\}}, 
	\mathbf{1}_{\{(X_i)_{i=l-1}^{l+t} \in u_nA\}}\right) \displaybreak[0]  \\
	%={}& \frac{1}{(n-t) \PP(X_0>u_n)^2} \left(\Var(\mathbf{1}_{\{(X_i)_{i=-1}^{t} \in u_nA\}}) + 2\sum_{h=1}^{n-t-1} \left(1 - \frac h {n-t}\right)    
	%\Cov(\mathbf{1}_{\{(X_i)_{i=-1}^{t} \in u_nA\}},\mathbf{1}_{\{(X_i)_{i=h-1}^{h+t} \in u_nA\}})\right) \\
	\leq{}& \frac{2}{n \PP(X_0>u_n)^2} \sum_{h=0}^{n-t-1}
	\Big| \PP((X_i)_{i=-1}^{t} \in u_n A, (X_i)_{i=h-1}^{h+t} \in u_n A) \\
	& \hspace{4cm} -  \PP((X_i)_{i=-1}^{t} \in u_n A) \PP((X_i)_{i=h-1}^{t} \in u_n A)  \Big|\displaybreak[0]  \\
	\leq{}& \frac{2}{n \PP(X_0 > u_n)} \sum_{h=0}^{r_n} \Big| \PP((X_i)_{i=-1}^{t} \in u_n A, (X_i)_{i=h-1}^{h+t} \in u_n A \mid X_0 > u_n) \\
	& \hspace{3cm} -  \PP((X_i)_{i=-1}^{t} \in u_n A \mid X_0 > u_n)   
	\PP((X_i)_{i=h-1}^{t} \in u_n A)  \Big| \\   
	& + \frac{2}{n \PP(X_0 > u_n)^2} \sum_{h=r_n+1}^{\infty} \alpha_{h-t-2} \displaybreak[0]  \\
	\leq{}& \frac{2 (r_n +1)}{n \PP(X_0 > u_n)} + \frac{2 C }{ n \PP(X_0 > u_n)^2} \sum_{h=r_n-t-1}^\infty h^{-\delta}\\
	\sim{}& \frac{2 (r_n +1)}{n \PP(X_0 > u_n)} + \frac{2 C (r_n - t - 1)^{1-\delta}}{(\delta-1) n \PP(X_0 > u_n)^2}.     
	\end{align*}	
	Setting $r_n = \PP(X_0 > u_n)^{-1/\delta}$, the right-hand side is asymptotically equal to
	$$ \frac{2 (1 + C/(\delta-1)) + o(1)}{n \PP(X_0 > u_n)^{1+1/\delta}}
	= \frac{2 (1 + C/(\delta-1)) + o(1)}{[n^{\delta/(1+\delta)} \PP(X_0 > u_n)]^{1+1/\delta}} \to 0. $$
	Thus, by Chebychev's inequality, this implies that
	$$ \frac{\widehat{P}_{u_n,n}(A) - \EE(\widehat{P}_{u_n,n}(A))}{\PP(X_0 > u_n)} \to_p 0.$$
	Since, by regular variation, 
	$$\frac{\EE(\widehat{P}_{u_n,n}(A))}{\PP(X_0 > u_n)} = \PP( (X_i)_{i=-1}^t \in u_n A \mid X_0 > u_n) \to \mu_{\{-1,\ldots,t\}}(A),
	\qquad n \to \infty,$$
	we obtain that
	$ \widehat{P}_{u_n,n}(A)/\PP(X_0 > u_n) \to_p \mu_{\{-1,\ldots,t\}}(A)$ both for $A=A_0$ and $A=A_1$.
	An application of the continuous mapping theorem for convergence in probability completes the proof. 
\end{proof}

\begin{proof}[Proof of Theorem \ref{thm:asymptotic-normality}]
	We prove the equivalent statement that all linear combinations of the random vector converge 
	in distribution to a centered normal distribution with the corresponding variance. 
	To this end, let $a_0, \ldots, a_N \in \RR$ and define
	$$ Z_{n,k} = \frac 1 {\sqrt{n \PP(X_0>u_n)}} \sum_{j=0}^N a_j \left( \mathbf{1}_{\{(X_i)_{i=k-1}^{k+t} \in u_n A_j\}} - \PP((X_i)_{i=-1}^t \in u_n A_j)  \right).$$
	We note that, for each $n \in \NN$, the random variable $\sum_{k=1}^l 
	Z_{n,k}$ is centered and that its variance converges
	\begin{align*}
	& \Var\left(\sum\nolimits_{k=1}^n Z_{n,k}\right) \\
	={}&  \sum_{j=0}^N \sum_{l=0}^N a_j a_l \sum_{h=-(n-1)}^{n-1} \frac{n-|h|}{n} \cdot \frac{\Cov(\mathbf{1}_{\{(X_i)_{i=-1}^t \in u_n A_j\}},\mathbf{1}_{\{(X_i)_{i=h-1}^{h+t} \in u_n A_l\}})}{\PP(X_0 > u_n)}\\
	\to{}& \sum\nolimits_{j=0}^N \sum\nolimits_{l=0}^N a_j a_l \sigma_{jl} 
	\end{align*}	
	to the desired quantity, which can be shown analogously to the proof of Lemma \ref{lem:variance} (i.e.\ the proof of Thm.~3.1 in
	\citet{davis-mikosch-09}).\\ 
	It remains to show that the asymptotic distribution is normal. To this end, we verify that the triangular scheme $\{Z_{n,k}\}_{k=1,\ldots,n}$, $n \in \NN$, satisfies
	the conditions of Thm.~4.4 in \citet{rio-17}: 
	\begin{itemize}
	\item At first, all the variables $Z_{n,k}$ are required to be centered and have finite variance which holds true as they are bounded.
	\item Secondly, we need to verify that
	  \begin{align} \label{eq:var-uniform}
		\limsup_{n \to \infty} \max_{l=1,\ldots,n} \Var\left(\sum\nolimits_{k=1}^l Z_{n,k}\right)  < \infty,
	  \end{align}
	  which again can be shown analogously to the proof of Lemma \ref{lem:variance}.
    \item The third and last condition to be verified is
	\begin{equation} \label{eq:rio}
	\lim_{x \to \infty} n \int_0^1 \alpha^{-1}(x) Q_{n,a}^2(x) \min\{ \alpha^{-1}(x) Q_{n,a}(x), 1\} \, \mathrm{d} x = 0
	\end{equation}
	where $\alpha^{-1}$ and $Q_{n,a}$ denote the inverse functions of $h \mapsto \alpha_h$
	and $$u \mapsto \PP( |Z_{n,k}| > u),$$ respectively.
	Noting that $Q_{n,a}$ can be bounded via the relation
	$$ Q_{n,a}(x) \leq \sum\nolimits_{j=1}^N |a_j| Q_{n,e_j}\left(\frac x N\right), \quad x \in [0,1], $$
	where $e_j$ is the $j$th standard basis vector in $\RR^N$, $j=1,\ldots,N$,
	it suffices to verify
	\begin{equation*} 
	\lim_{x \to \infty} n \int_0^1 \alpha^{-1}(x) Q_{n,e_j}^2\left(\frac x N\right) \min\left\{ \alpha^{-1}(x) Q_{n,e_j}\left(\frac x N\right), 1\right\} \, \mathrm{d} x = 0
	\end{equation*}
	for $j=1,\ldots,N$. By Equation \eqref{eq:alpha-decay}, we have $\alpha^{-1}(x) \leq (x/C)^{-1/\delta}$ for some $C>0$ and
	\begin{align*}
	Q_{n,e_j}\left(\frac x N\right) = \begin{cases} 
	\frac{1 - \PP( (X_i)_{i=-1}^t \in u_n A_j)}{\sqrt{n \PP(X_0 > u_n)}}, 
	& 0 \leq x < N \PP( (X_i)_{i=-1}^t \in u_n A_j), \\
	\frac{\PP( (X_i)_{i=-1}^t \in u_n A_j)}{\sqrt{n \PP(X_0 > u_n)}}, 
	& N \PP( (X_i)_{i=-1}^t \in u_n A_j) \leq x \leq 1,
	\end{cases}                
	\end{align*}
	for sufficiently large $n$. Thus, we have
	\begin{align*}
	& n \int_0^1 \alpha^{-1}(x) Q_{n,e_j}^2\left(\frac x N\right) \min\left\{ \alpha^{-1}(x) Q_{n,e_j}\left(\frac x N\right), 1\right\} \, \mathrm{d} x
	\displaybreak[0] \\
	\leq{}&    n \int_0^{N \PP((X_i)_{i=-1}^t \in u_n A_j)} (x/C)^{-1/\delta} (n \PP(X_0>u_n))^{-1}  \\
	& \hspace{3cm} \min\{1,(x/C)^{-1/\delta} (n \PP(X_0>u_n)^{-1/2}\} \, \mathrm{d} x  	\displaybreak[0] \\
	&  + n \int_{N \PP((X_i)_{i=-1}^t \in u_n A_j)}^1 (x/C)^{-2/\delta} \left(\frac{\PP((X_i)_{i=-1}^t \in u_n A_j)}{\sqrt{n \PP(X_0 > u_n)}}\right)^3  \, \mathrm{d} x \displaybreak[0] \\
	=:{}& I_1 + I_2.
	\end{align*}
	For the assessment of the integral term $I_2$, we employ the upper bound
	$$ \left(\frac x C\right)^{-2/\delta} \leq (C/N)^{2/\delta} \PP((X_i)_{i=-1}^t \in u_n A_j)^{-2/\delta} $$
	and obtain
	\begin{align*}
	I_2 \leq{}& n \int_0^1  \left(\frac{C}{N}\right)^{\frac{2}{\delta}}
	\PP((X_i)_{i=-1}^t \in u_n A_j)^{-\frac{2}{\delta}}  \left(\frac{\PP((X_i)_{i=-1}^t \in u_n A_j)}{\sqrt{n \PP(X_0 > u_n)}}\right)^3  \, \mathrm{d} x \\    
	\leq{}&  \left(\frac{C}{N}\right)^{\frac{2}{\delta}} \frac 1 {\sqrt{n}} \PP((X_i)_{i=-1}^t \in u_n A_j)^{\frac{3\delta-4}{2\delta}}
	\left( \frac{\PP((X_i)_{i=-1}^t \in u_n A_j)}{\PP(X_0 > u_n)}\right)^{3/2} \longrightarrow 0
	\end{align*}
	using that $\PP((X_i)_{i=-1}^t \in u_n A_j) / \PP(X_0 > u_n) \to \mu_{\{-1,\ldots,t\}}(A_j) < \infty$ 
	and $\PP((X_i)_{i=-1}^t \in u_n A_j) \to 0$ as $n \to \infty$.
	
	For the assessment of the integral $I_1$, we distinguish between the two cases $\delta=2$ and $\delta>2$.\\
	In the case $\delta=2$, we have
	\begin{align*}
	I_1 \leq{}& n \int_0^{N [ n \PP((X_i)_{i=-1}^t \in u_n A_j) ]^{-1}} C^{1/2}  x^{-1/2} (n \PP(X_0 > u_n))^{-1} \, \mathrm{d} x \\
	& + n \int_{N [ n \PP((X_i)_{i=-1}^t \in u_n A_j) ]^{- 1}}^{N \PP((X_i)_{i=-1}^t \in u_n A_j)} C  x^{-1} \sqrt{n \PP(X_0 > u_n)}^{-3} \, \mathrm{d}x \\        
	={}&    \sqrt{2CN} [ n \PP((X_i)_{i=-1}^t \in u_n A_j) ]^{-1/2} \PP(X_0 > u_n)^{-1} \\ 
	&  + C n^{-1/2} \PP(X_0 > u_n)^{-3/2} \log(\PP((X_i)_{i=-1}^t \in u_n A_j)) \\
	&  + C n^{-1/2} \PP(X_0 > u_n)^{-3/2} \log( n \PP((X_i)_{i=-1}^t \in u_n A_j) ) \\
	={}&    \sqrt{2CN}  n^{-1/2} \PP(X_0 >u_n)^{-3/2} \cdot \sqrt{\frac{\PP(X_0 > u_n)}{\PP((X_i)_{i=-1}^t \in u_n A_j)}} \\
	&  + C n^{-1/2} \PP(X_0 > u_n)^{-3/2} \cdot  \bigg\{ - \log(\PP(X_0 > u_n)) + \log( n \PP(X_0 > u_n)^3 ) \\ 
	&  \hspace{5cm} + 2 \log\left(\frac{\PP((X_i)_{i=-1}^t \in u_n A_j)}{\PP(X_0 > u_n)}\right) \bigg\} 
	\end{align*}
	which vanishes as $n \to \infty$ because of \eqref{eq:add-cond-delta0} and
	$\lim_{x \to \infty} x^{-1/2} \log(x) = 0$.\\
	For $\delta > 2$, we obtain 
	\begin{align*}
	I_1 \leq{}& n \int_{0}^{N \PP((X_i)_{i=-1}^t \in u_n A_j)}
	C^{2/\delta}  x^{-2/\delta} \sqrt{n \PP(X_0 > u_n)}^{-3} \, \mathrm{d}x \displaybreak[0]\\        
	={}& \frac{\delta}{\delta-2} N^{1-\frac{2}{\delta}} C^{\frac{2}{\delta}} \frac{1}{\sqrt{n}} \PP((X_i)_{i=-1}^t \in u_n A_j)^{1 - \frac{2}{\delta}} \PP(X_0 > u_n)^{-3/2} \displaybreak[0]\\
	={}& \frac{\delta}{\delta-2} N^{1-\frac{2}{\delta}} C^{\frac{2}{\delta}} \frac{1}{\sqrt{n}} \PP(X_0 > u_n)^{-\frac 12 \left(\frac{4+\delta}{\delta}\right)} \left(\frac{\PP((X_i)_{i=-1}^t \in u_n A_j)}{\PP(X_0 > u_n)}\right)^{1 - \frac{2}{\delta}}
	\end{align*}
	which vanishes as $n \to \infty$ because of \eqref{eq:add-cond}.
	\end{itemize} 
\end{proof}

	\begin{proof}[Proof of Proposition \ref{prop:maxstable}]
	We verify the validity of Equation \eqref{eq:mixing} and
	Equation \eqref{eq:anticlustering}. To this end, we use the results of \citet{dombry-eyiminko-12} to give an upper bound for the mixing 
	coefficients $\alpha_h$ in \eqref{eq:def-alpha} in terms of extremal
	coefficients \citep[see also][]{davis-mikosch-zhao-13}:
	\begin{equation} \label{eq:alphabound}
	\alpha_h \leq 2 \sum\nolimits_{s_1 = -\infty}^0 \sum\nolimits_{s_2 =0}^\infty \left(2 - \theta(s_2-s_1+h)\right) = 2 \sum\nolimits_{s=0}^\infty (s+1) \left(2 - \theta(s+h)\right).
	\end{equation}
	Making use of the monotonicity of the function $h \mapsto \theta(h)$ on $\NN_0$, the series considered in Equation \eqref{eq:mixing} can thus be bounded by
	\begin{align*}
	\frac{1}{\PP(X_0 > u_n)} \sum\nolimits_{h=r_n}^\infty \alpha_h \leq  \frac{2}{\PP(X_0 > u_n)} \sum\nolimits_{h=0}^\infty (h+1)(h+2) \left(2 - \theta(h+r_n)\right).
	\end{align*}
	The asymptotic relation $\PP(X_0 > u_n) = 1 - \exp(-u_n^{-\alpha}) \sim u_n^{-\alpha}$ then yields that \eqref{eq:mixing-maxstable} is a sufficient condition for \eqref{eq:mixing}.
	
	In order to simplify condition \eqref{eq:anticlustering}, we first note that 
	$$ \PP(X_h > u_n \mid X_0 > u_n) = \frac{2(1-\exp(-u_n^{-\alpha})) - (1-\exp(-\theta(h)u_n^{-\alpha}))}{1 - \exp(-u_n^{-\alpha})}. $$
	Using the inequality $x - x^2/2 \leq 1 - \exp(-x) \leq x$ for all $x \geq0$, 
	we obtain the bounds
	\begin{align} \label{eq:bounds-extremogram}
	& 2 - \theta(h) - \frac{u_n^{-\alpha}}2 \leq{}  \PP(X_h > u_n \mid X_0 > u_n) \nonumber \\
	\leq{}& \frac{2 - \theta(h) + \frac{\theta(h)^2 u_n^{-\alpha}}{2}}{1 - \frac {u_n^{-\alpha}} 2 } 
	\sim  2 - \theta(h) +  \frac{\theta(h)^2}{2} u_n^{-\alpha} \leq 2 - \theta(h) + 2 u_n^{-\alpha}.
	\end{align}
	As the sequences $r_n, u_n \to \infty$ are chosen such that $r_n \PP(X_0 > u_n) \to 0$ as $n \to \infty$,
	we obtain that, for every $k \in \NN$,
	$$ \sum\nolimits_{h=k}^{r_n}  u_n^{-\alpha} \leq \frac{r_n}{u_n^\alpha} = \frac{r_n \PP(X_0 > u_n)}{u_n^\alpha \PP(X_0 > u_n)} \to 0 \quad (n \to \infty),$$
	and consequently, from the inequalities above, \eqref{eq:anticlustering}
	holds if and only if \eqref{eq:anticlustering-maxstable} holds.
	
\end{proof}

\begin{proof}[Proof of Corollary \ref{coro:br}]
	We first note that, for the Brown--Resnick model,
	$$\theta(h) = 2 \Phi\left(\sqrt{\gamma(h)/2}\right), \qquad h \in \ZZ,$$
	cf.~\citet{kabluchko-etal-09}. In particular, using $1 - \Phi(x) \sim x^{-1} \varphi(x)$ as $x \to \infty$, we 
	obtain for large $h$ that
	$$ 2 - \theta(h) \sim  \frac{2}{\sqrt{\pi \gamma(h)}} \exp\left(-\frac{\gamma(h)}{4}\right). $$ 
	Thus, similarly to \citet{davis-mikosch-zhao-13}, we can employ Equation \eqref{eq:alphabound} to bound the $\alpha$-mixing coefficients by
	\begin{align*}
	\alpha_h \leq{}& \frac{4}{\sqrt{\pi C}} \sum\nolimits_{s=0}^\infty \frac{s+1}{(s+h)^{\varepsilon/2}}
	\exp\left(-\frac 1 4 C \cdot (s+h)^\varepsilon\right)\\
	\leq{}& \frac{4}{\sqrt{\pi C}} \sum\nolimits_{s=0}^\infty \frac{s+1}{s^{\varepsilon/2}}
	\exp\left(-\frac 1 8 C s^\varepsilon - \frac 1 8 C h^\varepsilon\right) 
	\leq C_\alpha \exp\left(- \frac 1 8 C h^\varepsilon\right)
	\end{align*}
	for some appropriate constant $C_\alpha > 0$, that is, the $\alpha$-mixing
	coefficients $\alpha_h$ decay at an exponential rate. 
	As Equation \eqref{eq:alpha-decay} holds for every $\delta > 0$,  Equation \eqref{eq:add-cond} simplifies to the condition $\lim_{n \to \infty} n^{1+\varepsilon^*} \PP(X_0 > u_n) 
	= \lim_{n \to \infty} n^{1-\varepsilon^*}/u_n = \infty$ for some 
	$\varepsilon^* > 0$, which is true for $u_n \sim n^{\beta_1}$.
	Thus, the assumptions of Theory \ref{thm:asymptotic-normality} reduce to Condition (M). 
	\medskip
	
	Choosing $u_n \sim n^{\beta_1}$ and $r_n \sim n^{\beta_2}$
	\citep[see also][]{buhl-klueppelberg-18,buhl-etal-19}, we have $u_n \to \infty$, $r_n \to \infty$, 
	$n \PP(X_0 > u_n) \sim n^{1 - \beta_1} \to \infty$ and 
	$r_n \PP(X_0 > u_n) \sim n^{\beta_2 - \beta_1} \to 0$ as $n \to \infty$.
	Furthermore, similarly to the assessment above,
	\begin{align*}
	& u_n \sum\nolimits_{h=1}^\infty h^2 \left(2 - \theta(h+r_n)\right) \\
	\leq{}& u_n \frac 4 {\sqrt{\pi C}} \exp\left(- \frac 1 8 C r_n^\varepsilon \right)
	\cdot \sum\nolimits_{h=1}^\infty h^{2-\varepsilon/2} \exp\left(- \frac 1 8 C h^\varepsilon \right)
	\stackrel{n \to \infty}{\longrightarrow} 0,
	\end{align*}
	i.e.\ Equation \eqref{eq:mixing-maxstable} holds. We also obtain
	$$ \sum_{h=k}^{r_n} \left( 2 - \theta(h)\right) 
	\leq \exp\left(- \frac 1 8 C (k-1)^\varepsilon\right) 
	\sum_{h=1}^\infty \frac{4}{\sqrt{\pi C h^\varepsilon}} \exp\left(- \frac 1 8 C h^\varepsilon\right) 
	\stackrel{k \to \infty}{\longrightarrow} 0 $$
	which implies \eqref{eq:anticlustering-maxstable}.
	Consequently, the assertion of the corollary follows from
	Proposition \ref{prop:maxstable}. 	
\end{proof}

\end{document}